\documentclass[12pt]{amsart}
\pdfoutput=1  
\usepackage{amssymb,amsmath,latexsym,amsthm}
\usepackage{fullpage}
\usepackage{mathpazo}
\usepackage[colorlinks, 
            linkcolor=MidnightBlue, 
            citecolor=MidnightBlue,
            urlcolor =BlueViolet,
           ]{hyperref}

\usepackage[usenames,dvipsnames]{xcolor} 
\usepackage{booktabs}                    
\usepackage[all]{xy}

\newtheorem{theorem}{Theorem}
\newtheorem{lemma}[theorem]{Lemma}
\newtheorem{corollary}[theorem]{Corollary}

\theoremstyle{remark}
\newtheorem{remark}[]{Remark}

\newcommand\CC{{\mathbb C}}
\newcommand\FF{{\mathbb F}}
\newcommand\Fq{\FF_q}
\newcommand\Ftwo{\FF_2}
\newcommand\RR{{\mathbb R}}
\newcommand\ZZ{{\mathbb Z}}

\newcommand{\ph}{\phantom{-}}
\newcommand{\zb}{\overline{z}}

\makeatletter
\@namedef{subjclassname@2020}{%
  \textup{2020} Mathematics Subject Classification}
\makeatother

\title[Abelian varieties over ${\mathbb F}_2$]
      {Every positive integer is the order of \\ 
       an ordinary abelian variety over ${\mathbb F}_2$}
\date{25 June 2021}

\author[Howe]{Everett W. Howe}
\address[Howe]{Unaffiliated mathematician, 
         San Diego, CA 92104, USA}
\email{however@alumni.caltech.edu}

\author[Kedlaya]{Kiran S. Kedlaya}
\address[Kedlaya]{Department of Mathematics, 
                  University of California San Diego, 
                  La Jolla, CA 92093, USA}
\email{kedlaya@ucsd.edu}

\thanks{Kedlaya was supported by NSF grant DMS-1802161 and the UCSD Warschawski Professorship.}


                 
\begin{document}

\begin{abstract}
We show that for every integer $m > 0$, there is an ordinary abelian variety
over~${\mathbb F}_2$ that has exactly $m$ rational points.
\end{abstract}

\maketitle

\section{Introduction}
\label{sec:intro}

The purpose of this paper is to prove the statement enunciated in its title,
namely that for every integer $m>0$ there is an ordinary abelian variety over
$\Ftwo$ that has exactly $m$ points over~$\Ftwo$. More specifically, we prove 
the following.

\begin{theorem}
\label{T:main}
Let $m > 0 $ and $d > 2$ be integers with $m < (4/3)2^d + 1.$ Then there is a
squarefree ordinary abelian variety $A$ over $\Ftwo$ of dimension at most $d$
with $m = \# A(\Ftwo)$.
\end{theorem}

(We say that an abelian variety is \emph{squarefree} if its decomposition up to
isogeny as a product $\prod B_i^{e_i}$ of powers of mutually non-isogenous
simple abelian varieties does not include any factors with $e_i > 1$.)

The supersingular elliptic curve $E/\Ftwo$ given by $y^2 + y = x^3 + x + 1$
satisfies $\#E(\Ftwo) = 1$. If we have an $n$-dimensional abelian variety
$A/\Ftwo$ with $m = \# A(\Ftwo)$, then by considering products $A \times E^e$ we
see that there are abelian varieties over $\Ftwo$ of every dimension greater
than or equal to $n$ that have $m$ points. As $A \times E^e$ is not simple when
$e > 0$, this leads us to a question of Kadets~\cite{Kadets2021}: For a given
positive integer~$m$, do there exist infinitely many \emph{simple} abelian
varieties $A$ over $\Ftwo$ with $\#A(\Ftwo) = m$? For $m=1$ the answer is known
to be \emph{yes}, thanks to the classification of such varieties given by Madan
and Pal~\cite{MadanPal1977}; moreover, infinitely many of these varieties are
ordinary. Resolving this question for $m>1$ probably requires a better
understanding of the space of Weil polynomials; see Section~\ref{sec:comments}.

One might ask whether there is a result analogous to Theorem~\ref{T:main} for 
other finite fields. Since the number of rational points on a $d$-dimensional
abelian variety over $\Fq$ lies in the interval
$[(\sqrt{q}-1)^{2d}, (\sqrt{q}+1)^{2d}]$, we see that for $q\ge 7$ the integer
$m = 2$ is too small to be the group order of an abelian variety over $\Fq$ of 
positive dimension, and it is not equal to the order of the unique
$0$-dimensional abelian variety over~$\Fq$.

On the other hand, building on our techniques, van Bommel, Costa, Li, Poonen,
and Smith \cite{vanBommelCostaEtAl2021} have shown that for $q \leq 5$, every
positive integer \emph{is} the order of an abelian variety over $\Fq$. The same
holds for ordinary abelian varieties with a single exception: for $q=4$ the
order $3$ cannot occur.

The work of van Bommel et al.~also shows that for all $q$, every
\emph{sufficiently large} positive integer is the order of an abelian variety
over~$\Fq$, which can also be taken to be ordinary, geometrically simple, and
principally polarizable. However, for $q > 2$ requiring both ordinariness and
geometric simplicity makes it impossible to achieve \emph{all} orders. Combining
\cite[Theorem~3.2]{Kadets2021} with data from the LMFDB \cite{LMFDB}, one can
prove the following:
\begin{itemize}
\item There is no geometrically simple ordinary abelian variety over $\FF_3$
      with $4$ points.
\item There is no geometrically simple ordinary abelian variety over $\FF_4$
      with $7$ points. 
\item There is no geometrically simple ordinary abelian variety over $\FF_5$
      with $6$ points.
\end{itemize}
We do not know whether every positive integer $m$ is the order of an ordinary,
geometrically simple, principally polarizable abelian variety over $\Ftwo$. One
way to check this would be to make the ``sufficiently large'' condition
effective (this is doable in principle but not carried out in
\cite{vanBommelCostaEtAl2021}), and then to tabulate isogeny classes of abelian
varieties over $\Ftwo$ as far as is needed to close the gap. (The LMFDB
currently contains all isogeny classes of abelian varieties over $\Ftwo$ of
dimension at most~$6$).

One can also ask about group \emph{structures} rather than group \emph{orders}.
Addressing this question requires additional ideas, because the group structure
of an abelian variety over a finite field is not an isogeny invariant. Using
Theorem~\ref{T:main} as a starting point, Marseglia and
Springer~\cite{MarsegliaSpringer2021} show that every finite abelian group is
isomorphic to $A(\FF_2)$ for some ordinary abelian variety $A$ over~$\FF_2$.

\subsubsection*{Acknowledgments}
The authors thank Francesc Fit\'e for questions and observations that led to
this research, Bjorn Poonen for his short clear proof of 
Lemma~\ref{L:EH condition}, and Stefano Marseglia and Caleb Springer for
observing that our construction produces squarefree varieties.

\section{Proof of the theorem}
\label{sec:proof}

The \emph{Weil polynomial} of an abelian variety $A$ over a finite field $\Fq$
is the characteristic polynomial of the Frobenius endomorphism of~$A$, acting,
say, on the $\ell$-adic Tate modules of~$A$. The Weil polynomial $f$ of $A$ lies
in $\ZZ[x]$, and if the dimension $n$ of $A$ is positive, then $f$ has the shape
\begin{equation}
\label{EQ:WP}
f = x^{2n} + a_1 x^{2n-1} + \cdots + a_{n-1} x^{n+1} + a_n x^n
     + q a_{n-1} x^{n-1} + \cdots + a_1 q^{n-1} x +  q^n.
\end{equation}
Furthermore, all of the complex roots of $f$ lie on the circle 
$\lvert z\rvert = \sqrt{q}$. The variety $A$ is \emph{ordinary} if $a_n$ is 
coprime to~$q$. Conversely, every polynomial $f\in\ZZ[x]$ of the 
shape~\eqref{EQ:WP} that has all of its roots on the circle 
$\lvert z\rvert = \sqrt{q}$ and with $a_n$ coprime to $q$ is the Weil polynomial
of an $n$-dimensional ordinary abelian variety over $\Fq$; this follows from the
Honda--Tate classification of Weil 
polynomials~\cite[Th\'eor\`eme~1, p.~96]{Tate1971}. This classification also
shows that an ordinary Weil polynomial is squarefree as a polynomial if and only
if the associated isogeny class is squarefree, in the sense defined in the
introduction.

The number of points on the variety $A$ is given by $f(1)$, so to prove
Theorem~\ref{T:main} we would like to have a large supply of Weil polynomials at
our disposal so that we can find or construct one whose value at~$1$ is a given
integer~$m$. There are a number of papers that give results that can be used to 
produce such polynomials --- see for example 
\cite{
vanBommelCostaEtAl2021,
Chen1995, 
DiPippoHowe1998,
Kwon2011, 
LakatosLosonczi2004, 
LakatosLosonczi2007, 
LakatosLosonczi2009, 
Schinzel2005, 
SinclairVaaler2008}. We will use Lemma~3.3.1 (p.~447)
of~\cite{DiPippoHowe1998}; for the convenience of the reader we reprove this
result here, as Corollary~\ref{C:EH condition}, using an argument suggested by
Bjorn Poonen. To state the result in terms helpful to us, we introduce some
notation.

Let $q$ be a prime power. If $(a_1,\ldots,a_n)$ is a finite sequence of
real numbers, we define the \emph{$q$-weight} $w_q((a_1,\ldots,a_n))$ of the
sequence to be the sum
\begin{equation}
\label{EQ:weight}
w_q((a_1,\ldots,a_n)) = \biggl\lvert \frac{a_n}{2 q^{n/2}}\biggr\rvert 
                     + \sum_{i=1}^{n-1} \, \biggl\lvert \frac{a_i}{q^{i/2}} \biggr\rvert.
\end{equation}

\begin{lemma} 
\label{L:EH condition}
Let $q$ be a prime power and let $(a_1,\ldots,a_n)$ be a sequence of real
numbers with $q$-weight at most~$1$. Then all of the complex roots of the
polynomial $f$ given by~\eqref{EQ:WP} lie on the circle 
$\lvert z\rvert = \sqrt{q}$. Moreover, if the $q$-weight of the sequence is
strictly less than~$1$, then the roots of $f$ are distinct, so that $f$ is
squarefree.
\end{lemma}

\begin{proof}
By continuity, it suffices to prove the statement when 
$w_q((a_1,\ldots,a_n)) < 1$.

Consider the meromorphic function $g$ defined for $z\in\CC$ by 
$g(z) = f(z) / z^n$. We would like to show that $g$ has $2n$ distinct zeros on
the circle $\lvert z\rvert = \sqrt{q}$. For $z$ on this circle, we have
\begin{equation}
\label{EQ:circle}
g(z) = (z^{n} + \zb^n) + a_1 (z^{n-1} + \zb^{n-1}) + \cdots 
       + a_{n-1} (z + \zb) + a_n,
\end{equation}
so we see that $g$ is real-valued on this circle. Consider $z\in\CC$ of the form
$z = \sqrt{q}\,\zeta$, where $\zeta^{2n} = 1$. For such a $z$, the initial term 
$(z^{n} + \zb^n)$ in~\eqref{EQ:circle} is equal to $\pm2q^{n/2}$, where the sign
is equal to the value of $\zeta^n$. The absolute value of the sum of the other
terms in~\eqref{EQ:circle} is bounded above by 
$2q^{n/2} w_q((a_1,\ldots,a_n)) < 2q^{n/2}$, so $g(z)$ and $(z^{n} + \zb^n)$
have the same sign. Thus, we have $n$ points on the circle where the value of
$g$ is positive interlaced with $n$ points where the value is negative, so $g$
must have $2n$ distinct zeros on the circle.
\end{proof}

\begin{corollary}[{\cite[Lemma~3.3.1, p.~447]{DiPippoHowe1998}}] 
\label{C:EH condition}
Let $q$ be a prime power and let $(a_1,\ldots,a_n)$ be a sequence of integers
with $q$-weight at most~$1$ and with $a_n$ coprime to~$q$. Then the polynomial
$f$ given by~\eqref{EQ:WP} is the Weil polynomial of an $n$-dimensional ordinary
abelian variety over~$\Fq$. Moreover, if the $q$-weight of the sequence is
strictly less than~$1$, then the abelian variety will be squarefree.
\qed
\end{corollary}

We will say that a sequence $(a_1,\ldots,a_n)$ of integers \emph{represents} an
integer $m$ if we have
\begin{equation}
\label{EQ:rep}
m = (2^n + 1) + a_1(2^{n-1} + 1) + \cdots + a_{n-1}(2^1 + 1) + a_n.
\end{equation}
This condition means exactly that $m = f(1)$ for the polynomial $f$ given
in~\eqref{EQ:WP}, with $q=2$.

We see that to prove Theorem~\ref{T:main}, it will suffice to show that if $m>0$
is an integer less than $(4/3)2^d + 1$ for some $d>2$, then $m$ is represented
by a sequence $(a_1,\ldots,a_n)$ of integers of length at most $d$ and of
$2$-weight less than~$1$, and with $a_n$ odd: for if this is the case we will
have $m = f(1)$ for the polynomial $f$ given in~\eqref{EQ:WP} with $q=2$, and
$f$ will be the Weil polynomial of a squarefree ordinary abelian variety over
$\Ftwo$ by Corollary~\ref{C:EH condition}. Thus, the following lemma completes
the proof of Theorem~\ref{T:main}.

\begin{lemma}
\label{L:representation}
Let $m > 0$ and $d > 2$ be integers such that $m < (4/3)2^d + 1$. Then $m$ is
represented by a sequence $(a_1,\ldots,a_n)$ of length $n\le d$ and $2$-weight
less than $1$, and with $a_n$ odd.
\end{lemma}

\begin{proof}
We prove the lemma by induction on $d$. In Table~\ref{table} we give, for each
$m\le 22$, a sequence $(a_1, \ldots, a_n)$ of $2$-weight less than $1$ and with
$a_n$ odd that represents~$m$. We observe that the lengths of these sequences
are all at most $4$, and for $m\le 11$ the lengths are at most~$3$. Since
$\lfloor (4/3)2^3 + 1\rfloor = 11$ and $\lfloor (4/3)2^4 + 1\rfloor = 22$, this
proves the statement for $d = 3$ and $d = 4$. 

\begin{table}
\centering
\begin{tabular}{r r r c c r r r c}
\toprule
$m$ & $n$ &                 $a_1, \ldots, a_n$ & $2$-weight & \hbox to 2em{} & 
$m$ & $n$ & \hbox to 1em{}  $a_1, \ldots, a_n$ & $2$-weight \\
\cmidrule{1-4}\cmidrule{6-9}
 1 & 2 &         -1,   -1 & 0.957 & & 12 & 3 & 0, \ph0, \ph0,   -5 & 0.625 \\
 2 & 1 &               -1 & 0.354 & & 13 & 4 & 0, \ph0,   -1,   -1 & 0.479 \\
 3 & 2 &         -1, \ph1 & 0.957 & & 14 & 4 & 0, \ph0, \ph0,   -3 & 0.375 \\
 4 & 2 &       \ph0,   -1 & 0.250 & & 15 & 4 & 0, \ph0,   -1, \ph1 & 0.479 \\
 5 & 3 & \ph0,   -1,   -1 & 0.677 & & 16 & 4 & 0, \ph0, \ph0,   -1 & 0.125 \\
 6 & 2 &       \ph0, \ph1 & 0.250 & & 17 & 4 & 0, \ph0,   -1, \ph3 & 0.729 \\
 7 & 3 & \ph0,   -1, \ph1 & 0.677 & & 18 & 4 & 0, \ph0, \ph0, \ph1 & 0.125 \\
 8 & 3 & \ph0, \ph0,   -1 & 0.177 & & 19 & 4 & 0, \ph0, \ph1,   -1 & 0.479 \\
 9 & 2 &       \ph1, \ph1 & 0.957 & & 20 & 4 & 0, \ph0, \ph0, \ph3 & 0.375 \\
10 & 3 & \ph0, \ph0, \ph1 & 0.177 & & 21 & 4 & 0, \ph0, \ph1, \ph1 & 0.479 \\
11 & 3 & \ph0, \ph1,   -1 & 0.677 & & 22 & 4 & 0, \ph0, \ph0, \ph5 & 0.625 \\
\bottomrule
\end{tabular}
\vskip2ex
\caption{For each $m\le 22$ we give a sequence $(a_1,\ldots, a_n)$ that
represents $m$ as in~\eqref{EQ:rep}, together with its $2$-weight, rounded to
three decimal places.}
\label{table}
\end{table}

Now suppose the statement of the lemma is true for all $d < D$, where $D\ge 5$,
and consider an integer $m$ with $m < (4/3)2^{D} + 1$. If $m < (4/3)2^{D-1} + 1$
then the conclusion of the lemma is true by the induction hypothesis, so we may
assume that $m > {(4/3)2^{D-1} + 1} = {(2/3)2^D + 1}$. Then
\[
\lvert m - (2^D + 1) \rvert < (1/3) 2^D  = (4/3)2^{D-2}.
\]
If $m = 2^D + 1$ then $m$ has a representation of length $D$ given by
$(0,\ldots,0,1,-3)$, and the $2$-weight of this sequence is
$1/2^{(D-1)/2} + 3/2^{(D+2)/2} < 1$. If $m \neq 2^D + 1$, then by the induction
hypothesis $\lvert m- (2^D + 1)\rvert$ is represented by a sequence 
$(b_1,\ldots,b_n)$ with $n \le D-2$, with $2$-weight less than~$1$, and with
$b_n$ odd. Let $s =\pm 1$ be the sign of $m- (2^D + 1)$. Then $m$ is represented
by the length-$D$ sequence $(0, \ldots, 0, s, s b_1, \ldots, s b_n)$, where the
initial $s$ occurs at position $D - n\ge 2$. We compute that then
\begin{align*}
w_2((0, \ldots, 0, s, s b_1, \ldots, s b_n)) 
  &= \frac{1}{2^{(D-n)/2}} \bigl( 1 + w_2(( b_1, \ldots,  b_n))\bigr)\\[1ex]
  &< (1/2)(1 + 1) = 1.
\end{align*}
This completes the induction and proves the lemma.
\end{proof}

\section{Remarks}
\label{sec:comments}

\begin{remark}
When $m \ge 10$, the sequence $(a_1,\ldots,a_n)$ produced by the proof of 
Lemma~\ref{L:representation} satisfies $a_1 = 0$, $\lvert a_i\rvert \le 1$ for
all $i<n$, and $a_i a_{i+1} = 0$ for all $i < n-2$. Thus, the
representation~\eqref{EQ:rep} of $m$ by the sequence $(a_1,\ldots,a_n)$ is
closely related in spirit to the signed binary representations of integers
described in \cite[\S 8]{Reitwiesner1960}; these are commonly known as
\emph{balanced binary representations} or 
\emph{NAF \textup(non-adjacent form\textup) representations}, particularly in
literature on efficient arithmetic in elliptic curve cryptography (for 
example,~\cite{JoyeTymen2001}).
\end{remark}

\begin{remark}
Using a more na\"{\i}ve construction of sequences, DiPippo and
Howe~\cite[Exercise~3.3.2, p.~450]{DiPippoHowe1998} show that for $d>1$, every
integer $m$ with $\lvert m - (2^d+1)\rvert\le (7/64)2^d$ is represented by a
sequence of length $d$ and $2$-weight at most~$1$. Our use of the non-adjacent
form construction allows us to replace $7/64$ with $1/3$ in this inequality when
$d>3$, and this allows us to cover all integers.
\end{remark}

\begin{remark}
One obstacle to determining which integers occur as the group orders of abelian
varieties over $\Fq$ is the difficulty of parametrizing the Weil polynomials of
$n$-dimensional abelian varieties over $\Fq$. The coefficient space of these
Weil polynomials can be viewed as a set of lattice points inside an explicitly
given region $V_n$ of $\RR^n$ that is homeomorphic to a simplex; however, as $n$
increases the simplex is stretched in ways that make it more and more difficult
to analyze the lattice points it contains. See~\cite{DiPippoHowe1998} for a
discussion of these regions.
\end{remark}


\nocite{DiPippoHowe2000}
\bibliography{AVF2}
\bibliographystyle{hplaindoi}

\end{document}